\documentclass[11pt]{article}
\usepackage{amsmath}
\usepackage{amsthm}
\usepackage{amsfonts}
\usepackage{tikz-cd}
\usepackage[utf8x]{inputenc}
\usepackage{enumerate}
\usepackage[english]{babel}
\usepackage{amssymb}
\usepackage{esint}
\usepackage{hyperref}
\usepackage{bm}
\usepackage{bbm}

\theoremstyle{plain}
\newtheorem{theorem}{Theorem}[section]
\newtheorem{lemma}[theorem]{Lemma}
\newtheorem{corollary}[theorem]{Corollary}

\theoremstyle{definition}
\newtheorem{definition}[theorem]{Definition}
\theoremstyle{remark}

\newtheorem{remark}[theorem]{Remark}

\numberwithin{equation}{section} 

\DeclareMathOperator{\divv}{div}

\DeclareMathOperator*{\supp}{supp}

\title{Weak-strong uniqueness principle\\for compressible barotropic self-gravitating fluids}
\author{Danica Basari\'{c} \thanks{The work of the author was funded from the Czech Science Foundation (GA\v{C}R), Grant Agreement 21-02411S. The institute of Mathematics of the Czech Academy of Sciences is supported by RVO:67985840.}}
\date{}

\begin{document}

\maketitle

\begin{center}
	Institute of Mathematics of the Czech Academy of Sciences \\
	\v{Z}itn\'{a} 25, 115 67 Praha 1, Czech Republic
\end{center}

\begin{center}
	E-mail address: basaric@math.cas.cz
\end{center}

\begin{abstract}
	The aim of this work is to prove the weak--strong uniqueness principle for the compressible Navier--Stokes--Poisson system on an exterior domain, with an isentropic pressure of the type $p(\varrho)=a\varrho^{\gamma}$ and allowing the density to be close or equal to zero. In particular, the result will be first obtained for an adiabatic exponent $\gamma \in [9/5,2]$ and afterwards, this range will be slightly enlarged via pressure estimates ``up to the boundary", deduced relaying on boundedness of a proper singular integral operator. 
\end{abstract}

\textbf{Mathematics Subject Classification:} 35J05, 35L65, 76N06

\vspace{0.2 cm}

\textbf{Keywords:} gaseous star; Navier--Stokes--Poisson system; weak-strong uniqueness principle

\section{Introduction}

In this paper we consider the compressible Navier--Stokes--Poisson system, characterized by the following equations:
\begin{equation} \label{the Navier-Stokes-Poisson system}
\begin{aligned}
\partial_t \varrho + \divv_x (\varrho \textbf{u}) &=0, \\
\partial_t(\varrho \textbf{u}) + \divv_x(\varrho \textbf{u}\otimes \textbf{u})+ \nabla_x p(\varrho) &= \divv_x \mathbb{S}(\nabla_x \textbf{u}) + G\varrho \nabla_x \Phi, \\
\sigma\Delta_x \Phi &= \varrho+g.
\end{aligned}
\end{equation}
The system will be studied $(0,T) \times \Omega$, where $T>0$ can be chosen arbitrarily large and $\Omega \subseteq \mathbb{R}^3$ is a bounded or unbounded domain. Here, the unknown variables are the density $\varrho=\varrho(t,x)$, the velocity $\textbf{u}=\textbf{u}(t,x)$ and the potential $\Phi=\Phi(t,x)$, while $p=p(\varrho)$ represents the barotropic pressure, $\mathbb{S}=\mathbb{S}(\nabla_x \textbf{u})$ the viscous stress tensor, which we suppose to be a linear function of the velocity gradient, $G$ is a positive constant and $g=g(x)$ a given function; further details can be found in Section \eqref{The system}.

Depending on the choice of $\sigma = \pm 1$, system \eqref{the Navier-Stokes-Poisson system} models two different physical phenomena:
\begin{itemize}
	\item for $\sigma=1$, it describes the transportation of charged particles in electronic devices and $(\varrho, \textbf{u}, \Phi)$ represent the density, velocity and electrostatic potential of the charge, respectively (see \cite{Deg} for more details);
	\item for $\sigma=-1$, it describes the motion of a gaseous star and $(\varrho, \textbf{u}, \Phi)$ represent the density, velocity and gravitational potential of the star, respectively.
\end{itemize}

In view of its importance in many real world problems, the Navier-Stokes-Poisson system \eqref{the Navier-Stokes-Poisson system} is a matter of great interest in mathematics and physics. Unfortunately, well-posedness of strong solutions was achieved only on a small time interval and for initial data satisfying some compatibility conditions, see for instance the work of Tan and Zhang \cite{TanZha}. On the other hand, something more can be said if we turn our attention to the class of weak solutions. For $\sigma=1$, the existence of global-in-time weak solutions was established on a bounded domain $\Omega$ and for a barotropic pressure of the type $p(\varrho) = a\varrho^{\gamma}$ by Donatelli \cite{Don} with the adiabatic exponent $\gamma \geq 3$, and by Kobayashi and Suzuki \cite{KobSuz} for $\gamma > \frac{3}{2}$, while on the whole space $\Omega=\mathbb{R}^3$ it was proved by Li, Matsumura and Zhang \cite{LiMatZha}. For $\sigma=-1$, the existence of global-in-time weak solutions was proved on an exterior domain $\Omega$ and with a barotropic pressure of the type $p(\varrho) = a\varrho^{\gamma}$, $\gamma>\frac{3}{2}$, by Ducomet and Feireisl \cite{DocFei}; later on, this result was improved for a non-monotone pressure by Ducomet, Feireisl, Petzeltov\'{a} and Stra\v{s}kraba \cite{DocFeiPetStr}. 

In this context, a bridge between the classes of strong and weak solutions can be constructed by means of an important analytical tool known as \textit{weak-strong uniqueness principle}: a weak solution of problem \eqref{the Navier-Stokes-Poisson system} coincides with the strong one, emanating from the same initial data, as long as the latter exists. The rather standard procedure in order to prove it is to introduce a positive functional measuring the ``distance" between the weak and strong solutions and to show that it vanishes for any time as a consequence of Gronwall's Lemma. The functional in question is known as \textit{relative energy} since it can be seen as a generalization of the mechanical energy associated to the system. However, the choice of $\sigma \in \{ 1,-1\}$ in the third equation of the Navier-Stokes-Poisson system \eqref{the Navier-Stokes-Poisson system} plays a key role in making the whole problem easier or more difficult, respectively. Indeed, notice that multiplying the second equation of system \eqref{the Navier-Stokes-Poisson system} by $\textbf{u}$, integrating over $(0,T) \times \Omega$ and imposing suitable boundary conditions for $\textbf{u}$ and $\nabla_x \Phi$, we can recover the energy inequality associated to the system:
\begin{equation*}
\left[\int_{\Omega} \left(\frac{1}{2} \varrho |\textbf{u}|^2 +P(\varrho)+ \frac{\sigma}{2}\ G\  |\nabla_x\Phi|^2  \right)(t,\cdot) \textup{d}x \right]_{t=0}^{t=\tau} + \int_{0}^{\tau} \int_{\Omega} \mathbb{S}(\nabla_x \textbf{u}):\nabla_x \textbf{u} \ \textup{d}x \textup{d}t\leq 0, 
\end{equation*}
where $P=P(\varrho)$ denotes the pressure potential; further details can be found in Section \ref{Energy inequality}.

For $\sigma=1$, it then makes sense to consider the relative energy functional as
\begin{equation} \label{relative energy}
\mathcal{E}\left(\varrho, \textbf{u}, \Phi \ | \ \widetilde{\varrho}, \widetilde{\textbf{u}}, \widetilde{\Phi}\right) = \int_{\Omega} \left(\frac{1}{2} \varrho |\textbf{u}-\widetilde{\textbf{u}}|^2+ P(\varrho)- P'(\widetilde{\varrho}) (\varrho-\widetilde{\varrho}) -P(\widetilde{\varrho}) +\frac{1}{2} G \ |\nabla_x(\Phi- \widetilde{\Phi})|^2\right) \textup{d}x,
\end{equation}
where $(\varrho, \textbf{u}, \Phi)$ and $(\widetilde{\varrho}, \widetilde{\textbf{u}}, \widetilde{\Phi})$ denote the weak and strong solutions of system \eqref{the Navier-Stokes-Poisson system}, respectively. Indeed, the convexity of the pressure potential $P=P(\varrho)$ guarantees the non-negativity of $\mathcal{E}=\mathcal{E}(t)$ for any time $t\in [0,T]$. Moreover, if $\varrho, \widetilde{\varrho}>0$, proving the weak-strong uniqueness principle is equivalent to showing that $\mathcal{E}(t)\equiv 0$ for any time $t\in [0,T]$; this is the strategy pursued by He and Tan \cite{HeTan} to prove the weak-strong uniqueness principle on a bounded domain $\Omega$.

For $\sigma=-1$, however, the problem gets more complicated. First of all, the analogous of \eqref{relative energy} would be 
\begin{equation*}
\mathcal{E}\left(\varrho, \textbf{u}, \Phi \ | \ \widetilde{\varrho}, \widetilde{\textbf{u}}, \widetilde{\Phi}\right) = \int_{\Omega} \left(\frac{1}{2} \varrho |\textbf{u}-\widetilde{\textbf{u}}|^2+ P(\varrho)- P'(\widetilde{\varrho}) (\varrho-\widetilde{\varrho}) -P(\widetilde{\varrho}) -\frac{1}{2} G \ |\nabla_x(\Phi- \widetilde{\Phi})|^2\right) \textup{d}x,
\end{equation*}
but with this choice we cannot guarantee the non-negativity of $\mathcal{E}=\mathcal{E}(t)$ for any time $t\in [0,T]$. Moreover, for $\sigma=-1$, the Navier-Stokes-Poisson system \eqref{the Navier-Stokes-Poisson system} describes the motion of a gaseous star and thus the optimal choice for $\Omega$ is to be exterior to a rigid object; however, working on an unbounded domain prevents us from using some useful tools such as the Sobolev-Poincar\'{e} inequality. A third difficulty is represented by the fact that the density is close to zero, at least in the far field, and therefore we loose the strict positivity of $\varrho, \widetilde{\varrho}$. 

To handle these problems for $\sigma=-1$, first of all we will consider the relative energy to be a function of the density and velocity only, cf. Section \ref{Relative eneergy inequality}. Indeed, it is well--known that the solution $\Phi$ of the Poisson equation 
\begin{equation*}
-\Delta_x \Phi =f
\end{equation*}
on the whole space $\mathbb{R}^3$ is uniquely determined by the corresponding known term $f$. Therefore, in our context it makes sense to write $\Phi = (-\Delta_x)^{-1} (\varrho+g)$, provided $\varrho$ can be extended to be zero outside $\Omega$, and, as a consequence of the H\"{o}rmander-Mikhlin Theorem, we will be able to recover some useful estimates for $\nabla_x \Phi$ depending on the density only, which will be fundamental in proving the weak--strong uniqueness principle, cf. Section \ref{Weak-strong uniqueness}. The problem of the vanishing strong solution $\widetilde{\varrho}$ can be handled following the same idea developed by  Feireisl and Novotn\'{y} in \cite{FeiNov}, considering first $\widetilde{\varrho}+ \varepsilon$, with $\varepsilon>0$, instead of $\widetilde{\varrho}$ in the relative energy functional to get a strictly positive quantity and passing to the limit $\varepsilon \rightarrow 0$. In particular, in \cite{FeiNov} the authors were able to prove the weak--strong uniqueness principle for a general compressible viscous fluid on an exterior domain and with a barotropic pressure of the type $p(\varrho)=a \varrho^{\gamma}$ with $1<\gamma \leq 2$. In our context, the presence of the gravitational potential forces the range for the adiabatic exponent to be 
\begin{equation*}
	\frac{9}{5} \leq \gamma \leq 2,
\end{equation*}
where, in particular, the lower bound coincides with the critical exponent appearing in the book of Lions \cite{Lio1}. However, the result can be improved if we manage to get better regularity for the density. This will be achieved deducing pressure estimates ``up to the boundary", obtained adapting the work of Feireisl and Petzeltov\'{a} in \cite{FeiPet} for a Lipschitz exterior domain and exploiting, in particular, the boundedness of the singular operator $\nabla_x (-\Delta_x)^{-1} \nabla_x$, cf. Section \ref{Pressure estimates up to the boundary}.

The work is organized as follows. Section \ref{The system} will be devoted to the detailed description of the system we are going to study, deducing, in particular, the energy inequality associated to it. In Section \ref{Dissipative weak solution}, we provide the definition of a dissipative weak solution, cf. Definition \ref{dissipative solution}, while in Section \ref{Relative eneergy inequality}, we recover the relative energy inequality, cf. Lemma \ref{lemma relative energy inequality}. Section \ref{Weak-strong uniqueness} will be devoted to the proof of the weak--strong uniqueness principle, cf. Theorem \ref{weak-strong uniqueness}. Finally, in Section \ref{pressure estimate up to the boundary}, we are able to improve the result obtained in the previous section, cf. Corollary \ref{corollary}, by means of the pressure estimates ``up to the boundary", cf. Theorem \ref{pressure estimate}.

\section{The system} \label{The system}
We consider the Navier-Stokes-Poisson system, describing the motion of a gaseous star:
\begin{align} 
\partial_t \varrho + \divv_x (\varrho \textbf{u}) &=0, \label{continuity equation} \\
\partial_t(\varrho \textbf{u}) + \divv_x(\varrho \textbf{u}\otimes \textbf{u})+ \nabla_x p(\varrho) &= \divv_x \mathbb{S}(\nabla_x \textbf{u}) + G\varrho \nabla_x \Phi, \label{balance of momentum} \\
-\Delta_x \Phi &= \varrho+g. \label{Poisson equation}
\end{align}
Here, the unknown variables are the density $\varrho=\varrho(t,x)$, the velocity $\textbf{u}=\textbf{u}(t,x)$ and the gravitational potential $\Phi=\Phi(t,x)$ of the star. For simplicity, we assume an isentropic pressure $p=p(\varrho)$ of the type 
\begin{equation*}
p(\varrho)= a \varrho ^{\gamma}
\end{equation*}
for a constant $a>0$, with the adiabatic exponent
\begin{equation*}
\gamma > 1,
\end{equation*}
while the viscous stress tensor is a linear function of the velocity gradient, more specifically it satisfies Newton's rheological law
\begin{equation} \label{viscous stress tensor}
\mathbb{S}(\nabla_x \textbf{u})= \mu \left(\nabla_x \textbf{u}+\nabla_x^T \textbf{u}- \frac{2}{3}(\divv_x \textbf{u})\mathbb{I}\right) + \lambda(\divv_x \textbf{u})\mathbb{I},
\end{equation}
with $\mu >0$ and $\lambda\geq 0$. Finally, $G$ is a positive constant and $g=g(x)$ is a given function, which for simplicity we suppose to satisfy
\begin{equation*}
	g \in L^1 \cap L^{\infty}(\mathbb{R}^3).
\end{equation*}

We will study the system on $(0,T) \times \Omega$, where the time $T>0$ can be chosen arbitrarily large while $\Omega \subset \mathbb{R}^3$ is a Lipschitz exterior domain, on the boundary of which we impose
\begin{equation} \label{boundary conditions}
\textbf{u}|_{\partial \Omega}=0;
\end{equation}
moreover, we fix the conditions at infinity as
\begin{equation}
\varrho \rightarrow 0, \quad \textbf{u} \rightarrow 0 \quad \mbox{as } |x| \rightarrow \infty.
\end{equation}
The system is formally closed prescribing the initial conditions for the density and momentum:
\begin{equation} \label{initial conditions}
\varrho(0, \cdot)=\varrho_0, \quad (\varrho\textbf{u})(0,\cdot)= \textbf{m}_0. 
\end{equation}

\subsection{Poisson equation}

Noticing that the Poisson equation \eqref{Poisson equation} holds on the whole space $\mathbb{R}^3$, provided $\varrho$ is extended to be zero outside $\Omega$, we can write
\begin{equation*}
\Phi(t,x) = [\Gamma * (\varrho+g)] (t,x) = \int_{\mathbb{R}^3}  [\varrho(t,y)+ g(y) ] \ \Gamma(|x-y|) \ \textup{d}y 
\end{equation*}
where
\begin{equation*}
\Gamma(|x|) = \frac{1}{4\pi |x|}
\end{equation*}
is the \textit{fundamental solution} of the Laplace equation \eqref{Poisson equation}. Therefore, the gravitational potential $\Phi$ is uniquely determined by the corresponding density $\varrho$ and therefore it is not necessary to consider it as a third variable.

\subsection{Energy inequality} \label{Energy inequality}

Multiplying equation \eqref{balance of momentum} by $\textbf{u}$ and noticing that each term of this product can be rewritten as
\begin{equation} \label{energy calculations}
\begin{aligned}
\partial_t(\varrho \textbf{u})\cdot \textbf{u}&= \frac{\partial}{\partial t} \left(\frac{1}{2} \varrho |\textbf{u}|^2\right)+ \frac{1}{2} |\textbf{u}|^2 \partial_t \varrho, \\
\divv_x(\varrho \textbf{u} \otimes \textbf{u}) \cdot \textbf{u} &= \divv_x \left( \frac{1}{2} \varrho |\textbf{u}|^2 \textbf{u}\right) +\frac{1}{2} |\textbf{u}|^2  \divv_x(\varrho \textbf{u}), \\
\nabla_x p(\varrho)\cdot \textbf{u}&= \divv_x[p(\varrho)\textbf{u}]-p(\varrho)\divv_x \textbf{u}, \\
\divv_x \mathbb{S}(\nabla_x \textbf{u}) \cdot \textbf{u}&= \divv_x[\mathbb{S}(\nabla_x \textbf{u})\textbf{u}]- \mathbb{S}(\nabla_x \textbf{u}):\nabla_x \textbf{u}, \\
\varrho \nabla_x \Phi \cdot \textbf{u} &= \divv_x (\varrho \Phi \textbf{u}) - \Phi \divv_x(\varrho \textbf{u}) ,
\end{aligned}
\end{equation}
where, in particular, from \eqref{continuity equation} and \eqref{Poisson equation},
\begin{align*}
- \Phi \divv_x(\varrho \textbf{u}) &= \Phi \partial_t \varrho = -\Phi \partial_t \Delta_x \Phi= -\Phi \divv_x [\partial_t \nabla_x \Phi] \\
&= -\divv_x[\Phi \ \partial_t \nabla_x \Phi]+ \nabla_x \Phi \cdot \partial_t \nabla_x \Phi\\
&= -\divv_x[\Phi \ \partial_t \nabla_x \Phi]+ \frac{\partial}{\partial t} \left( \frac{1}{2}|\nabla_x \Phi|^2 \right),
\end{align*}
from the continuity equation \eqref{continuity equation}, we get
\begin{align*}
\frac{1}{2} \frac{\partial}{\partial t} \left(\varrho |\textbf{u}|^2 - G |\nabla_x \Phi|^2 \right)&+ \divv_x\left[ \left( \frac{1}{2} \varrho |\textbf{u}|^2 +p(\varrho) \right) \textbf{u} \right] -p(\varrho)\divv_x \textbf{u}  \\
&= \divv_x \left[ \big( \mathbb{S}(\nabla_x \textbf{u})+\varrho \Phi \big) \textbf{u} \right] - \mathbb{S}(\nabla_x \textbf{u}):\nabla_x \textbf{u} -\divv_x[\Phi \ \partial_t \nabla_x \Phi].
\end{align*}
Integrating over $\Omega$, keeping in mind that $\textbf{u}$ satisfies the boundary condition \eqref{boundary conditions} and imposing that
\begin{equation} \label{boundary condition gravitational potential}
\nabla_x \Phi \cdot \textbf{n}|_{\partial \Omega}=0,
\end{equation}
we get 
\begin{equation} \label{first energy equality}
\begin{aligned}
\frac{1}{2} \frac{\textup{d}}{\textup{d}t} \int_{\Omega} \left( \varrho |\textbf{u}|^2 -G |\nabla_x \Phi|^2  \right) \textup{d}x -\int_{\Omega} p(\varrho)\divv_x \textbf{u} \ \textup{d}x+\int_{\Omega} \mathbb{S}(\nabla_x \textbf{u}):\nabla_x \textbf{u} \ \textup{d}x =0.
\end{aligned}
\end{equation}
Introducing the \textit{pressure potential} $P=P(\varrho)$ as a solution of
\begin{equation} \label{pressure potential equation}
\varrho P'(\varrho) - P(\varrho) =p(\varrho),
\end{equation}
from the continuity equation \eqref{continuity equation}, we can write
\begin{equation*}
-p(\varrho) \divv_x \textbf{u} = \partial_t P(\varrho) + \divv_x [P(\varrho) \textbf{u}].
\end{equation*}
We finally get the \textit{energy inequality}
\begin{equation*}
\frac{\textup{d}}{\textup{d}t} \int_{\Omega} \left(\frac{1}{2} \varrho |\textbf{u}|^2 +P(\varrho)- \frac{1}{2} G|\nabla_x\Phi|^2  \right)(t,\cdot) \  \textup{d}x+ \int_{\Omega} \mathbb{S}(\nabla_x \textbf{u}):\nabla_x \textbf{u} \ \textup{d}x \leq 0.
\end{equation*}

Alternatively, we can leave the last term in \eqref{energy calculations} unchanged and get
\begin{equation} \label{energy inequality 1}
\frac{\textup{d}}{\textup{d}t} E(t) + \int_{\Omega} \mathbb{S}(\nabla_x \textbf{u}):\nabla_x \textbf{u} \ \textup{d}x \leq G  \int_{\Omega} \varrho \nabla_x \Phi \cdot \textbf{u} \ \textup{d}x, 
\end{equation}
with 
\begin{equation} \label{energy}
E(t):= \int_{\Omega} \left(\frac{1}{2} \varrho |\textbf{u}|^2 +P(\varrho)  \right) (t,\cdot) \ \textup{d}x.
\end{equation}
\begin{remark}
	Hereafter, we will consider \eqref{energy inequality 1}, \eqref{energy} and therefore we don't need the boundary condition \eqref{boundary condition gravitational potential} for the gravitation potential.
\end{remark}

\section{Dissipative weak solution} \label{Dissipative weak solution}

We are now ready to give the definition of a \textit{dissipative weak solution} to the compressible Navier-Stokes-Poisson system. Following the same definition presented in \cite{DocFeiPetStr}, a dissipative weak solution of problem \eqref{continuity equation}--\eqref{initial conditions} is a couple $[\varrho, \textbf{u}]$ such that
\begin{enumerate}
	\item equation \eqref{continuity equation} and its renormalized version hold in a distributional sense on the whole $(0,T) \times \mathbb{R}^3$, provided $\varrho$ and $\textbf{u}$ are extended to be zero outside $\Omega$;
	\item equation \eqref{balance of momentum} holds in a distributional sense on $(0,T) \times \Omega$;
	\item equation \eqref{Poisson equation} is satisfied a.e. on $\mathbb{R}^3$ for any fixed $t\in (0,T)$, provided $\varrho$ is extended to be zero outside $\Omega$; 
	\item  the integral version of the energy inequality \eqref{energy inequality 1} holds on $(0,T)$.
\end{enumerate}
More precisely, we have the following definition.

\begin{definition} \label{dissipative solution}
	The pair of functions $(\varrho, \textbf{u})$ is called \textit{dissipative weak solution} of the Navier-Stokes-Poisson system \eqref{continuity equation}--\eqref{initial conditions} with initial conditions $\varrho_0$, $\textbf{m}_0$ satisfying
	\begin{equation*}
	\varrho_0 \in L^1 \cap L^{\gamma}(\Omega), \quad \varrho_0\geq 0\mbox{ a.e. in } \Omega, \quad \frac{|\textup{\textbf{m}}_0|^2}{\varrho_0} \in L^1(\Omega).
	\end{equation*}
	if the following holds:
	\begin{itemize}
		\item[(i)] \textit{regularity class}: 
		\begin{align*}
		\varrho &\in C_{\textup{weak}}([0,T]; L^1 \cap L^{\gamma}(\Omega)), \\
		\varrho\textbf{u} &\in C_{\textup{weak}}([0,T]; L^{\frac{2\gamma}{\gamma+1}}(\Omega; \mathbb{R}^3)), \\
		\textbf{u}&\in L^2(0,T; D^{1,2}_0(\Omega; \mathbb{R}^3)), 
		\end{align*}
		and $\varrho$ is a non-negative function a.e. in $(0,T) \times \Omega$;
		\item[(ii)] \textit{weak formulation of the continuity equation}: for any $\tau \in (0,T)$, the integral identity
		\begin{equation} \label{weak formulation continuity equation}
		\left[ \int_{\Omega} \varrho \varphi(t,\cdot) 	\ \textup{d}x\right]_{t=0}^{t=\tau} = \int_{0}^{\tau} \int_{\Omega} [\varrho \partial_t\varphi+ \varrho\textbf{u}\cdot \nabla_x\varphi ] \ \textup{d}x\textup{d}t,
		\end{equation}
		holds for any $\varphi \in C_c^1([0,T]\times \mathbb{R}^3)$, with $\varrho(0,\cdot)= \varrho_0$, provided $\varrho$ and $\textbf{u}$ are extended to be zero outside $\Omega$;
		\item[(iii)] \textit{weak formulation of the renormalized continuity equation}: for any $\tau \in (0,T)$ and any functions
		\begin{equation*}
		B\in C[0,\infty) \cap C^1(0,\infty), \ b\in C[0,\infty) \mbox{ bounded on } [0,\infty),
		\end{equation*}
		\begin{equation*}
		B(0)=b(0)=0 \quad \mbox{and} \quad b(z)=zB'(z)-B(z) \mbox{ for any }z>0,
		\end{equation*}
		the integral identity
		\begin{equation} \label{weak formulation renormalized continuity equation}
		\left[ \int_{\Omega} B(\varrho) \varphi(t,\cdot) \  \textup{d}x\right]_{t=0}^{t=\tau} = \int_{0}^{\tau} \int_{\Omega} [B(\varrho) \partial_t\varphi+ B(\varrho)\textbf{u}\cdot \nabla_x\varphi + b(\varrho)\divv_x \textbf{u} \varphi] \ \textup{d}x\textup{d}t,
		\end{equation}
		holds for any $\varphi \in C_c^1([0,T]\times\mathbb{R}^3)$, provided $\varrho$ and $\textbf{u}$ are extended to be zero outside $\Omega$;
		\item[(iv)] \textit{weak formulation of the balance of momentum}: for any $\tau \in (0,T)$, the integral identity
		\begin{equation} \label{weak formulation of the balance of momentum}
		\begin{aligned}
		\left[ \int_{\Omega} \varrho \textbf{u} \cdot \bm{\varphi}(t,\cdot) \ \textup{d}x\right]_{t=0}^{t=\tau} &= \int_{0}^{\tau} \int_{\Omega} [\varrho \textbf{u}\cdot \partial_t\bm{\varphi}+ (\varrho \textbf{u} \otimes \textbf{u}): \nabla_x \bm{\varphi}+ p(\varrho)\divv_x \bm{\varphi}] \ \textup{d}x\textup{d}t , \\
		&-\int_{0}^{\tau} \int_{\Omega} \mathbb{S}(\nabla_x \textbf{u}): \nabla_x \bm{\varphi} \ \textup{d}x\textup{d}t + G \int_{0}^{\tau} \int_{\Omega} \varrho \nabla_x \Phi \cdot \bm{\varphi} \ \textup{d}x\textup{d}t
		\end{aligned}
		\end{equation}
		holds for any $\bm{\varphi} \in C_c^1([0,T]\times \overline{\Omega}; \mathbb{R}^3)$, $\bm{\varphi}|_{\partial \Omega}=0$, with $(\varrho \textbf{u})(0,\cdot)= \textbf{m}_0$;
		\item[(v)] \textit{Poisson equation}: for any fixed $t \in [0,T]$, equation \eqref{Poisson equation} is satisfied a.e. on $\mathbb{R}^3$, provided $\varrho$ is extended to be zero outside $\Omega$;
		\item[(vi)] \textit{energy inequality}: inequality
		\begin{equation} \label{energy inequality}
		\,[E(t)]_{t=0}^{t=\tau} + \int_{0}^{\tau} \int_{\Omega} \mathbb{S}(\nabla_x \textbf{u}):\nabla_x \textbf{u} \ \textup{d}x \textup{d}t \leq G \int_{0}^{\tau} \int_{\Omega} \varrho \nabla_x \Phi \cdot \textbf{u} \ \textup{d}x \textup{d}t,
		\end{equation}
		holds for a.e. $\tau \in (0,T)$, with 
		\begin{equation*}
		E(t)= \int_{\Omega} \left(\frac{1}{2} \varrho |\textbf{u}|^2 +P(\varrho)  \right) (t,\cdot) \ \textup{d}x.
		\end{equation*}
	\end{itemize}
\end{definition}

\begin{remark}
	Hereafter, we denote with $D_0^{1,p}(\Omega)$, $1\leq p <\infty$ the completion of $C_c^{\infty}(\Omega)$ with respect to the norm
	\begin{equation*}
	\| u\|_{D_0^{1,p}(\Omega)}= \| \nabla_x u\|_{L^p(\Omega)}.
	\end{equation*}
\end{remark}

\begin{remark}
	Notice that conditions (i) of Definition \ref{dissipative solution} come naturally from the assumption that the total mechanical energy of the system is bounded at the initial time $t=0$.
\end{remark}

\begin{remark} \label{less regularity test functions}
	By a density argument, the test functions in the weak formulations \eqref{weak formulation continuity equation}--\eqref{weak formulation of the balance of momentum} can be taken less regular as long as all the integrals remain well--defined.
\end{remark}

\section{Relative energy inequality} \label{Relative eneergy inequality}

The aim of this section is to prove that any dissipative weak solution $[\varrho, \textbf{u}]$ of the compressible Navier--Stokes--Poisson system \eqref{continuity equation}--\eqref{initial conditions} satisfies an extended version of the energy inequality, known as \textit{relative energy inequality}, for $[\widetilde{\varrho}, \widetilde{\textbf{u}}]$ regular enough. More precisely, our goal is to prove the following result.

\begin{lemma} \label{lemma relative energy inequality}
	Let $[\varrho, \textup{\textbf{u}}]$ be a dissipative weak solution of the compressible Navier-Stokes-Poisson system \eqref{continuity equation}--\eqref{initial conditions} in the sense of Definition \ref{dissipative solution}. Then for any pair of functions $[\widetilde{\varrho}, \widetilde{\textup{\textbf{u}}}]$ such that 
	\begin{align*}
	\widetilde{\varrho}&\in C^1([0,T]\times \overline{\Omega}), \\
	\widetilde{\textup{\textbf{u}}} &\in C^1_c([0,T]\times \overline{\Omega}; \mathbb{R}^3),
	\end{align*}
	with $\widetilde{\varrho}>0$, the following inequality holds:
	\begin{equation} \label{relative energy inequality}
	\begin{aligned}
	\left[ \int_{\Omega} E(\varrho, \textup{\textbf{u}} \ | \ \widetilde{\varrho}, \widetilde{\textup{\textbf{u}}}) (t, \cdot) \  \textup{d}x\right]_{t=0}^{t=\tau}&+ \int_{0}^{\tau} \int_{\Omega} \mathbb{S}(\nabla_x \textup{\textbf{u}}):\nabla_x (\textup{\textbf{u}}-\widetilde{\textup{\textbf{u}}}) \ \textup{d}x \textup{d}t\\
	\leq& -\int_{0}^{\tau} \int_{\Omega}  \varrho(\textup{\textbf{u}}-\widetilde{\textup{\textbf{u}}})  \cdot[\partial_t \widetilde{\textup{\textbf{u}}} + \nabla_x \widetilde{\textup{\textbf{u}}} \cdot \widetilde{\textup{\textbf{u}}}+ \nabla_x P'(\widetilde{\varrho})] \ \textup{d}x \textup{d}t \\
	& - \int_{0}^{\tau} \int_{\Omega} \varrho (\textup{\textbf{u}}-\widetilde{\textup{\textbf{u}}}) \cdot \nabla_x \widetilde{\textup{\textbf{u}}} \cdot ( \textup{\textbf{u}}-\widetilde{\textup{\textbf{u}}}) \ \textup{d}x \textup{d}t \\
	&-\int_{0}^{\tau} \int_{\Omega}  [p(\varrho)-p'(\widetilde{\varrho})(\varrho-\widetilde{\varrho})-p(\widetilde{\varrho})] \divv_x \widetilde{\textup{\textbf{u}}} \ \textup{d}x \textup{d}t \\
	&+\int_{0}^{\tau}\int_{\Omega} p'(\widetilde{\varrho}) \left( 1- \frac{\varrho}{\widetilde{\varrho}}\right) [\partial_t \widetilde{\varrho} + \divv_x (\widetilde{\varrho} \widetilde{\textup{\textbf{u}}})]  \  \textup{d}x \textup{d}t \\
	&+ G\int_{0}^{\tau}\int_{\Omega} \varrho (\textup{\textbf{u}}-\widetilde{\textup{\textbf{u}}}) \cdot \nabla_x \Phi \ \textup{d}x \textup{d}t
	\end{aligned}
	\end{equation}
	with
	\begin{equation*}
	E(\varrho,\textup{\textbf{u}} \ | \ \widetilde{\varrho}, \widetilde{\textup{\textbf{u}}}) = \frac{1}{2} \varrho |\textup{\textbf{u}}-\widetilde{\textup{\textbf{u}}}|^2+ P(\varrho)- P'(\widetilde{\varrho}) (\varrho-\widetilde{\varrho}) -P(\widetilde{\varrho}).
	\end{equation*}
\end{lemma}

\begin{remark}
	We define the \textit{relative energy} for any $\tau\in [0,T]$ as
	\begin{equation*}
	\mathcal{E}(\varrho,\textup{\textbf{u}} \ | \ \widetilde{\varrho}, \widetilde{\textup{\textbf{u}}}) (\tau)= \int_{\Omega} E(\varrho,\textup{\textbf{u}} \ | \ \widetilde{\varrho}, \widetilde{\textup{\textbf{u}}}) (\tau, \cdot) \  \textup{d}x ,
	\end{equation*}
	and consequently, we refer to relation \eqref{relative energy inequality} as \textit{relative energy inequality}.
\end{remark}

\begin{proof}
	First of all, we can take $\bm{\varphi}=\widetilde{\textbf{u}}$ in the weak formulation of the momentum equation \eqref{weak formulation of the balance of momentum} to obtain
	\begin{equation} \label{first equation for relative energy inequality}
	\begin{aligned}
	\left[ \int_{\Omega} \varrho \textbf{u} \cdot 	\widetilde{\textbf{u}}(t,\cdot) \ \textup{d}x\right]_{t=0}^{t=\tau} &= \int_{0}^{\tau} \int_{\Omega} [\varrho \textbf{u}\cdot \partial_t	\widetilde{\textbf{u}}+ (\varrho \textbf{u} \otimes \textbf{u}): \nabla_x \widetilde{\textbf{u}}+ p(\varrho)\divv_x \widetilde{\textbf{u}}] \ \textup{d}x\textup{d}t \\
	&-\int_{0}^{\tau} \int_{\Omega} \mathbb{S}(\nabla_x \textbf{u}): \nabla_x \widetilde{\textbf{u}} \ \textup{d}x\textup{d}t +G \int_{0}^{\tau} \int_{\Omega} \varrho \nabla_x \Phi \cdot \widetilde{\textbf{u}} \ \textup{d}x\textup{d}t
	\end{aligned}
	\end{equation}
	$\varphi=\frac{1}{2} |\widetilde{\textbf{u}}|^2, P'(\widetilde{\varrho})$ in the weak formulation of the continuity equation \eqref{weak formulation continuity equation} to get
	\begin{equation} \label{second equation for relative energy inequality}
	\left[ \int_{\Omega} \frac{1}{2}\varrho |	\widetilde{\textbf{u}}|^2(t,\cdot) 	\ \textup{d}x\right]_{t=0}^{t=\tau} = \int_{0}^{\tau} \int_{\Omega} [\varrho \widetilde{\textbf{u}} \cdot \partial_t 	\widetilde{\textbf{u}}+ \varrho\textbf{u}\cdot \nabla_x \widetilde{\textbf{u}} \cdot \widetilde{\textbf{u}} ] \ \textup{d}x\textup{d}t,
	\end{equation}
	\begin{equation} \label{third equation for relative energy inequality} 
	\left[ \int_{\Omega} \varrho P'(\widetilde{\varrho})(t,\cdot) 	\ \textup{d}x\right]_{t=0}^{t=\tau} = \int_{0}^{\tau} \int_{\Omega} [\varrho \partial_tP'(\widetilde{\varrho})+ \varrho\textbf{u}\cdot \nabla_x P'(\widetilde{\varrho}) ] \ \textup{d}x\textup{d}t,
	\end{equation}
	respectively.
	
	Subtracting equations \eqref{first equation for relative energy inequality}, \eqref{third equation for relative energy inequality} and summing equation \eqref{second equation for relative energy inequality} to the energy inequality \eqref{energy inequality}, we get 
	\begin{align*}
	&\left[ \int_{\Omega} \left(\frac{1}{2} \varrho |\textbf{u}-\widetilde{\textbf{u}}|^2+ P(\varrho)- P'(\widetilde{\varrho}) (\varrho-\widetilde{\varrho}) -P(\widetilde{\varrho})\right)  \textup{d}x\right]_{t=0}^{t=\tau} \\
	&\quad \quad + \int_{0}^{\tau} \int_{\Omega} \mathbb{S}(\nabla_x \textbf{u}):\nabla_x (\textbf{u}-\widetilde{\textbf{u}}) \ \textup{d}x \textup{d}t \\
	&\quad \quad \leq -\int_{0}^{\tau} \int_{\Omega}  \varrho(\textbf{u}-\widetilde{\textbf{u}})  \cdot[\partial_t \widetilde{\textbf{u}} + \nabla_x \widetilde{\textbf{u}} \cdot \widetilde{\textbf{u}} + \nabla_x P'(\widetilde{\varrho})] \ \textup{d}x \textup{d}t \\
	&\quad \quad - \int_{0}^{\tau} \int_{\Omega} \varrho ( \textbf{u}-\widetilde{\textbf{u}} ) \cdot \nabla_x \widetilde{\textbf{u}} \cdot ( \textbf{u}-\widetilde{\textbf{u}}) \ \textup{d}x \textup{d}t \\
	& \quad \quad -\int_{0}^{\tau} \int_{\Omega}  [p(\varrho)-p'(\widetilde{\varrho})(\varrho-\widetilde{\varrho})-p(\widetilde{\varrho})] \divv_x \widetilde{\textbf{u}} \ \textup{d}x \textup{d}t \\
	&\quad \quad + G\int_{0}^{\tau}\int_{\Omega} \varrho (\textup{\textbf{u}}-\widetilde{\textup{\textbf{u}}}) \cdot \nabla_x \Phi  \ \textup{d}x \textup{d}t \\
	& \quad \quad +\int_{0}^{\tau}\int_{\Omega} \left[ [p'(\widetilde{\varrho})-\varrho P''(\widetilde{\varrho})] \partial_t \widetilde{\varrho} +p'(\widetilde{\varrho})(\widetilde{\varrho}-\varrho) \divv_x \widetilde{\textbf{u}} +[p'(\widetilde{\varrho})-\varrho P''(\widetilde{\varrho})] \nabla_x \widetilde{\varrho} \cdot \widetilde{\textbf{u}} \right] \textup{d}x \textup{d}t
	\end{align*}
	where we summed and subtracted quantities $\varrho \widetilde{\textbf{u}} \cdot \nabla_x P'(\widetilde{\varrho})$, $p'(\widetilde{\varrho})(\varrho-\widetilde{\varrho})+p(\widetilde{\varrho})$ and $\partial_t p(\widetilde{\varrho})$. 
	
	Now, keeping in mind condition \eqref{pressure potential equation} and in particular the fact that $p'(\widetilde{\varrho})= \widetilde{\varrho} P''(\widetilde{\varrho})$, we can rewrite the last line as follows 
	\begin{align*}
	& \int_{0}^{\tau}\int_{\Omega} \left[ [p'(\widetilde{\varrho})-\varrho P''(\widetilde{\varrho})] \partial_t \widetilde{\varrho} +p'(\widetilde{\varrho})(\widetilde{\varrho}-\varrho) \divv_x \widetilde{\textbf{u}} +[p'(\widetilde{\varrho})-\varrho P''(\widetilde{\varrho})] \nabla_x \widetilde{\varrho} \cdot \widetilde{\textbf{u}} \right] \textup{d}x \textup{d}t \\
	=& \int_{0}^{\tau}\int_{\Omega} p'(\widetilde{\varrho}) \left( 1- \frac{\varrho}{\widetilde{\varrho}}\right) [\partial_t \widetilde{\varrho} + \widetilde{\varrho} \divv_x \widetilde{\textbf{u}} + \nabla_x \widetilde{\varrho} \cdot \widetilde{\textbf{u}}]  \  \textup{d}x \textup{d}t \\
	=& \int_{0}^{\tau}\int_{\Omega} p'(\widetilde{\varrho}) \left( 1- \frac{\varrho}{\widetilde{\varrho}}\right) [\partial_t \widetilde{\varrho} + \divv_x (\widetilde{\varrho} \widetilde{\textbf{u}})]  \  \textup{d}x \textup{d}t.
	\end{align*}
	We finally got \eqref{relative energy inequality}.
\end{proof}

\section{Weak-strong uniqueness} \label{Weak-strong uniqueness}
In this section our goal is to prove that a dissipative weak solution coincides with the strong one emanating from the same initial data, as long as the latter exists. The strategy consists in showing, through a standard Gronwall argument and relying on the relative energy inequality \eqref{relative energy inequality}, that $	\mathcal{E}(\varrho,\textup{\textbf{u}} \ | \ \widetilde{\varrho}, \widetilde{\textup{\textbf{u}}}) (\tau)$ vanishes for any $\tau \in [0,T]$, where $(\varrho, \textbf{u})$ and $(\widetilde{\varrho}, \widetilde{\textbf{u}})$ denote the weak and strong solutions, respectively. One difficulty is represented by the fact the $\widetilde{\varrho}=0$, at least in the far field, and therefore we cannot simply plug it in \eqref{relative energy inequality}. We can then follow the same idea developed by Feireisl and Novotn\'{y} \cite{FeiNov}, working with the couple $(\widetilde{\varrho}+\varepsilon, \widetilde{\textbf{u}})$, $\varepsilon>0$, and performing the limit $\varepsilon \rightarrow 0$. While in \cite{FeiNov} the result was proved for $1<\gamma \leq 2$, the presence of the gravitational potential in this context reduces the interval to $\gamma_1 \leq \gamma \leq 2$, with $\gamma_1=\frac{9}{5}$. In order to prove the main result of this section, we need the following lemma.

\begin{lemma} \label{estimate gradient potential}
	Let $\Phi$ be the solution of the Laplace equation
	\begin{equation*}
	-\Delta_x \Phi = f \quad \mbox{on } \mathbb{R}^3.
	\end{equation*}
	If $f \in L^p(\mathbb{R}^3)$, $1<p<3$, then
	\begin{equation} \label{bound gradient potential}
	\| \nabla_x \Phi \|_{L^{p^*}(\mathbb{R}^3; \mathbb{R}^3)} \leq c \| f\|_{L^p(\mathbb{R}^3)}
	\end{equation}
	with $c=c(p)$ a positive constant and $p^*$ the Sobolev conjugate of $p$,
	\begin{equation*}
	\frac{1}{p^*} = \frac{1}{p}- \frac{1}{3}.
	\end{equation*}
\end{lemma}
\begin{proof}
	Writing $\Phi= (-\Delta_x)^{-1}(f)$, from Sobolev's inequality we have that
	\begin{equation*}
	\| \nabla_x \Phi \|_{L^{p^*}(\mathbb{R}^3; \mathbb{R}^3)} = \left\| \nabla_x (-\Delta_x)^{-1}(f) \right\|_{L^{p^*}(\mathbb{R}^3; \mathbb{R}^3)} \leq c \left\| \nabla_x(-\Delta_x)^{-1}\nabla_x (f) \right\|_{L^{p}(\mathbb{R}^3; \mathbb{R}^{3\times 3})}.
	\end{equation*}
	Now, let $\mathcal{F}$,$\mathcal{F}^{-1}$ denote the Fourier transform and its inverse, respectively. We can write
	\begin{equation*}
	\nabla_x(-\Delta_x)^{-1} \nabla_x (f) = \left[ \frac{\partial^2}{\partial x_j \partial x_k} \mathcal{F}^{-1} \left( \frac{1}{|\xi|^2} \mathcal{F}(f) \right) \right]_{j,k=1}^{3} = \left[ \mathcal{F}^{-1} \left( -\frac{\xi_j \xi_k}{|\xi|^2} \mathcal{F}(f) \right) \right]_{j,k=1}^{3}.
	\end{equation*}
	It is easy to show that the multipliers $m_{jk}(\xi)= - \frac{\xi_j \xi_k }{|\xi|^2}$, $j,k=1,\dots,3$, satisfy the hypothesis of the H\"{o}rmander-Mikhlin Theorem (see \cite{Ste}, Chapter 4, Theorem 3). Hence the pseudo-differential operator $ \nabla_x(-\Delta_x)^{-1} \nabla_x $ is a bounded linear operator on $L^p(\mathbb{R})$ for any $1<p<\infty$ and therefore 
	\begin{equation} \label{boundedness singular operator}
	\left\| \nabla_x (-\Delta_x)^{-1} \nabla_x\ (f) \right\|_{L^{p}(\mathbb{R}^3; \mathbb{R}^{3\times 3})} \leq \| f\|_{L^p(\mathbb{R}^3)},
	\end{equation} 
	implying \eqref{bound gradient potential}.
\end{proof}

\begin{theorem} \label{weak-strong uniqueness}
	Let 
	\begin{equation} \label{condition gamma}
	\frac{9}{5} \leq \gamma \leq 2
	\end{equation}
	and let $(\varrho, \textbf{\textup{u}})$ be a dissipative weak solution of problem \eqref{continuity equation}--\eqref{initial conditions} in the sense of Definition \ref{dissipative solution}. Let $(\widetilde{\varrho}, \widetilde{\textup{\textbf{u}}})$ be a strong solution of the same problem such that
	\begin{equation} \label{regularity class}
		\begin{aligned}
			\widetilde{\textup{\textbf{u}}} &\in C([0,T]; D_0^{1,2} \cap D^{3,2}(\Omega; \mathbb{R}^3)) \cap L^2(0,T; D^{4,2}(\Omega; \mathbb{R}^3)) \\
			\partial_t \widetilde{\textup{\textbf{u}}} &\in L^{\infty} (0,T; D^{1,2}_0(\Omega; \mathbb{R}^3)) \cap L^2(0,T; D^{2,2}(\Omega; \mathbb{R}^3)), \\
			\widetilde{\varrho}, \ p(\widetilde{\varrho}) &\in C([0,T]; W^{3,2}(\Omega)).
		\end{aligned}
	\end{equation}
	Moreover, suppose that
	\begin{equation*} 
	\varrho(0, \cdot)=\widetilde{\varrho}(0,\cdot)=\varrho_0, \quad (\varrho\textup{\textbf{u}})(0,\cdot)= (\widetilde{\varrho} \widetilde{\textup{\textbf{u}}})(0,\cdot)= \textbf{\textup{m}}_0, 
	\end{equation*}
	where 
	\begin{equation} \label{initial density}
	\varrho_0 \in C_c(\overline{\Omega}), \quad (\varrho_0)^{\gamma-1} \in W^{1,6}\cap W^{1,\infty}(\Omega; \mathbb{R}^3).
	\end{equation}
	Then 
	\begin{equation} \label{equality weak and strong}
	\varrho\equiv \widetilde{\varrho}, \quad \textbf{\textup{u}} \equiv \widetilde{\textup{\textbf{u}}} \quad \mbox{in } (0,T) \times \Omega.
	\end{equation}
\end{theorem}

\begin{remark}
	We point out that the regularity class \eqref{regularity class} chosen for $(\widetilde{\varrho}, \widetilde{\textbf{u}})$ is the one introduced by Huang, Li and Xin \cite{HuaLiXin}. Moreover, the two conditions in \eqref{initial density} for the initial density $\varrho_0$ guarantee 
	\begin{align}
		\widetilde{\varrho}(t,x)=0& \quad \mbox{for any } t\in [0,T] \mbox{ and } |x|>R, \label{compactly supported density} \\
		\sup_{t\in(0,T)}&\| \nabla_x(\widetilde{\varrho}^{\ \gamma-1})\|_{L^6\cap L^{\infty}(\Omega; \mathbb{R}^3)} \lesssim 1, \label{regularity gradient pressure potential}
	\end{align}
	respectively, for $R>0$ sufficiently large. Indeed, from Lemma 2.1 in \cite{FeiNov} we deduce that if the velocity field $\widetilde{\textbf{u}}$ is smooth enough, the regularity in \eqref{initial density} will propagate in time. Conditions \eqref{compactly supported density}, \eqref{regularity gradient pressure potential}, on the other hand, are fundamental in providing a proper bound for the term
	\begin{equation*}
		\partial_t \widetilde{\textbf{u}} + \nabla_x \widetilde{\textbf{u}} \cdot \widetilde{\textbf{u}} + \nabla_x P'(\widetilde{\varrho})- G \nabla_x \widetilde{\Phi};
	\end{equation*}
	see Sections 5.1 and 5.2 in \cite{FeiNov} for further details.
\end{remark}

\begin{proof}
	As we cannot plug $\widetilde{\varrho}$ in \eqref{relative energy inequality} since condition $\inf \widetilde{\varrho} >0$ is necessary, we can plug in $\widetilde{\varrho}+ \varepsilon$ to get
	\begin{equation*}
	\begin{aligned}
	\int_{\Omega} E(\varrho, \textbf{u} \ | \ \widetilde{\varrho}+\varepsilon, \widetilde{\textbf{u}}) (\tau, \cdot) \  \textup{d}x  &+ \int_{0}^{\tau} \int_{\Omega} \mathbb{S}(\nabla_x \textbf{u}):\nabla_x (\textbf{u}-\widetilde{\textbf{u}}) \ \textup{d}x \textup{d}t \\
	&\leq \int_{\Omega} \left( P(\varrho_0) +\varepsilon P'(\varrho_0 +\varepsilon) -P(\varrho_0+\varepsilon)\right) \textup{d}x\\
	&+ \int_{0}^{\tau} \int_{\Omega}  \varrho(\widetilde{\textbf{u}}-\textbf{u})  \cdot[\partial_t \widetilde{\textbf{u}} + \nabla_x \widetilde{\textbf{u}} \cdot \widetilde{\textbf{u}} + \nabla_x P'(\widetilde{\varrho}+\varepsilon)] \ \textup{d}x \textup{d}t \\
	&- \int_{0}^{\tau} \int_{\Omega} \varrho (\widetilde{\textbf{u}} - \textbf{u}) \cdot \nabla_x \widetilde{\textbf{u}} \cdot (\widetilde{\textbf{u}}- \textbf{u}) \ \textup{d}x \textup{d}t \\
	&-\int_{0}^{\tau} \int_{\Omega}  [p(\varrho)-p'(\widetilde{\varrho}+\varepsilon)(\varrho-\widetilde{\varrho}-\varepsilon)-p(\widetilde{\varrho} +\varepsilon)] \divv_x \widetilde{\textbf{u}} \ \textup{d}x \textup{d}t \\
	& + G\int_{0}^{\tau}\int_{\Omega} \varrho (\textbf{u}-\widetilde{\textbf{u}}) \cdot \nabla_x \Phi \  \textup{d}x \textup{d}t \\
	& + \varepsilon\int_{0}^{\tau}\int_{\Omega} p'(\widetilde{\varrho}+\varepsilon) \left( 1- \frac{\varrho}{\widetilde{\varrho}+\varepsilon}\right)  \divv_x  \widetilde{\textbf{u}}  \  \textup{d}x \textup{d}t.
	\end{aligned}
	\end{equation*}
	Repeating the same passages done in \cite{FeiNov} and performing the limit $\varepsilon \rightarrow 0$ we obtain
	\begin{equation*}
	\begin{aligned}
	\int_{\Omega} E(\varrho, \textbf{u} \ | \ \widetilde{\varrho}, \widetilde{\textbf{u}}) (\tau, \cdot) \  \textup{d}x & + \int_{0}^{\tau} \int_{\Omega} \mathbb{S}(\nabla_x \textbf{u}):\nabla_x (\textbf{u}-\widetilde{\textbf{u}}) \ \textup{d}x \textup{d}t  \\
	& \leq \int_{0}^{\tau}  \int_{\Omega} E(\varrho, \textbf{u} \ | \ \widetilde{\varrho}, \widetilde{\textbf{u}})  \  \textup{d}x \textup{d}t \\
	&+ \int_{0}^{\tau} \int_{\Omega}  \varrho(\widetilde{\textbf{u}}-\textbf{u})  \cdot[\partial_t \widetilde{\textbf{u}} + \nabla_x \widetilde{\textbf{u}} \cdot \widetilde{\textbf{u}} + \nabla_x P'(\widetilde{\varrho})-G \nabla_x \widetilde{\Phi}] \ \textup{d}x \textup{d}t \\
	&+G  \int_{0}^{\tau} \int_{\Omega} \varrho (\textbf{u}-\widetilde{\textbf{u}}) \cdot \nabla_x (\Phi -\widetilde{\Phi})\ \textup{d}x \textup{d}t
	\end{aligned}
	\end{equation*}
	As the couple $[\widetilde{\varrho}, \widetilde{\textbf{u}}]$ is a strong solution of our problem, it satisfies
	\begin{equation*}
	\widetilde{\varrho} \  [\partial_t \widetilde{\textbf{u}} + \nabla_x \widetilde{\textbf{u}} \cdot \widetilde{\textbf{u}} + \nabla_x P'(\widetilde{\varrho})- G\nabla_x \widetilde{\Phi}] = \divv_x \mathbb{S}(\nabla_x \widetilde{\textbf{u}}) 
	\end{equation*}
	and hence, we can add on both sides of the previous inequality the quantity
	\begin{align*}
	\int_{0}^{\tau} \int_{\Omega}  \widetilde{\varrho}\ (\textbf{u}-\widetilde{\textbf{u}})  \cdot[\partial_t \widetilde{\textbf{u}} + \nabla_x \widetilde{\textbf{u}} \cdot \widetilde{\textbf{u}} + \nabla_x P'(\widetilde{\varrho})- G \nabla_x \widetilde{\Phi}] \ \textup{d}x \textup{d}t =- \int_{0}^{\tau} \int_{\Omega} \mathbb{S}(\nabla_x \widetilde{\textbf{u}}):\nabla_x (\textbf{u}-\widetilde{\textbf{u}}) \ \textup{d}x \textup{d}t
	\end{align*}
	to get 
	\begin{equation*}
	\begin{aligned}
	\int_{\Omega} E(\varrho, \textbf{u} \ | \ \widetilde{\varrho}, \widetilde{\textbf{u}}) (\tau, \cdot) \  \textup{d}x & + \int_{0}^{\tau} \int_{\Omega} \mathbb{S}(\nabla_x (\textbf{u}-\widetilde{\textbf{u}})):\nabla_x (\textbf{u}-\widetilde{\textbf{u}}) \ \textup{d}x \textup{d}t\\ 
	& \leq \int_{0}^{\tau} \int_{\Omega} E(\varrho, \textbf{u} \ | \ \widetilde{\varrho}, \widetilde{\textbf{u}}) \  \textup{d}x \textup{d}t \\
	&+ \int_{0}^{\tau} \int_{\Omega}  (\varrho-\widetilde{\varrho})(\widetilde{\textbf{u}}-\textbf{u})  \cdot[\partial_t \widetilde{\textbf{u}} + \nabla_x \widetilde{\textbf{u}} \cdot \widetilde{\textbf{u}} + \nabla_x P'(\widetilde{\varrho}) -G \  \nabla_x \widetilde{\Phi}] \ \textup{d}x \textup{d}t \\
	&+G  \int_{0}^{\tau} \int_{\Omega}  \varrho (\textbf{u}-\widetilde{\textbf{u}}) \cdot \nabla_x (\Phi-\widetilde{\Phi}) \ \textup{d}x \textup{d}t
	\end{aligned}
	\end{equation*}
	From \eqref{viscous stress tensor} and keeping in mind that $\textbf{u}$ and $\widetilde{\textbf{u}}$ vanish on the boundary of $\Omega$, we have 
	\begin{align*}
	\int_{0}^{\tau} \int_{\Omega} & \mathbb{S}(\nabla_x (\textbf{u}-\widetilde{\textbf{u}})):\nabla_x (\textbf{u}-\widetilde{\textbf{u}}) \ \textup{d}x \textup{d}t \\ 
	&= \int_{0}^{\tau} \int_{\Omega} \left[ \mu |\nabla_x (\textbf{u}-\widetilde{\textbf{u}})|^2+ \left(\frac{1}{3}\mu +\lambda\right) |\divv_x (\textbf{u}-\widetilde{\textbf{u}})|^2 \right] \textup{d}x \textup{d}t \\
	&\geq \mu \int_{0}^{\tau} \int_{\Omega} |\nabla_x (\textbf{u}-\widetilde{\textbf{u}})|^2 \ \textup{d}x \textup{d}t
	\end{align*}
	and thus finally we can infer
	\begin{equation} \label{energy inequality 2}
	\begin{aligned}
	\int_{\Omega} E(\varrho, \textbf{u} \ | \ \widetilde{\varrho}, \widetilde{\textbf{u}}) (\tau, \cdot) \  \textup{d}x & + \int_{0}^{\tau} \int_{\Omega} |\nabla_x (\textbf{u}-\widetilde{\textbf{u}})|^2 \ \textup{d}x \textup{d}t\\ 
	& \lesssim \int_{0}^{\tau} \int_{\Omega} E(\varrho, \textbf{u} \ | \ \widetilde{\varrho}, \widetilde{\textbf{u}}) \  \textup{d}x  \textup{d}t \\
	&+ \int_{0}^{\tau} \int_{\Omega}  	(\varrho-\widetilde{\varrho})(\widetilde{\textbf{u}}-\textbf{u})  \cdot[\partial_t \widetilde{\textbf{u}} + \nabla_x \widetilde{\textbf{u}} \cdot \widetilde{\textbf{u}} + \nabla_x P'(\widetilde{\varrho}) -G \  \nabla_x \widetilde{\Phi}] \ \textup{d}x \textup{d}t \\
	&+G  \int_{0}^{\tau} \int_{\Omega}  \varrho (\textbf{u}-\widetilde{\textbf{u}}) \cdot \nabla_x (\Phi-\widetilde{\Phi}) \ \textup{d}x \textup{d}t.
	\end{aligned}
	\end{equation}
	Since the hypothesis of Theorem 5.2 in \cite{FeiNov} are satisfied, we can repeat the same passages to obtain
	\begin{equation*}
	\int_{\Omega} (\varrho-\widetilde{\varrho})(\widetilde{\textbf{u}}-\textbf{u})  \cdot[\partial_t \widetilde{\textbf{u}} + \nabla_x \widetilde{\textbf{u}} \cdot \widetilde{\textbf{u}} + \nabla_x P'(\widetilde{\varrho}) -G \  \nabla_x \widetilde{\Phi}](t, \cdot) \ \textup{d}x \lesssim \int_{\Omega} E(\varrho, \textbf{u} \ | \ \widetilde{\varrho}, \widetilde{\textbf{u}}) (t, \cdot) \  \textup{d}x,
	\end{equation*}
	as soon as 
	\begin{equation} \label{condition gamma 1}
	1<\gamma \leq 2.
	\end{equation}
	
	It remains to control the last term in \eqref{energy inequality 2}. First of all, notice that $-\Delta_x(\Phi-\widetilde{\Phi})= \varrho-\widetilde{\varrho}$ and thus we get
	\begin{align*}
	\int_{\Omega}  \varrho &\ (\textbf{u}-\widetilde{\textbf{u}}) \cdot \nabla_x (\Phi-\widetilde{\Phi}) \ \textup{d}x \\
	&= \int_{\Omega} \varrho \ (\textbf{u}-\widetilde{\textbf{u}}) \cdot \nabla_x(-\Delta_x)^{-1}(\varrho-\widetilde{\varrho}) \ \textup{d}x \\ 
	&\leq \frac{1}{2} \int_{\Omega} \varrho |\textbf{u}-\widetilde{\textbf{u}}|^2 \ \textup{d}x+ \frac{1}{2} \int_{\Omega} \varrho |\nabla_x(-\Delta_x)^{-1}(\varrho-\widetilde{\varrho})|^2 \ \textup{d}x,
	\end{align*}
	where the first term of the right-hand side of the inequality can be controlled by the relative energy and therefore it remains to estimate the second term. Fix $\bar{\varrho}>0$ so that $\widetilde{\varrho} \leq \frac{1}{2} \bar{\varrho}$; then, we obtain the following inequality
	\begin{align*}
	\int_{\Omega} \varrho& |\nabla_x(-\Delta_x)^{-1}(\varrho-\widetilde{\varrho})|^2 \ \textup{d}x \\
	&\leq \int_{\Omega} \varrho | \nabla_x(-\Delta_x)^{-1} \big( \mathbbm{1}_{\varrho\geq \bar{\varrho}} (\varrho-\widetilde{\varrho}) \big) |^2 \ \textup{d}x +  \int_{\Omega} \varrho | \nabla_x(-\Delta_x)^{-1} \big( \mathbbm{1}_{\varrho\leq \bar{\varrho}} (\varrho-\widetilde{\varrho}) \big) |^2 \ \textup{d}x,
	\end{align*}
	
	On one hand, from H\"{o}lder's inequality we get
	\begin{equation*}
	\int_{\Omega} \varrho | \nabla_x(-\Delta_x)^{-1}\big( \mathbbm{1}_{\varrho\geq \bar{\varrho}} (\varrho-\widetilde{\varrho}) \big) |^2 \ \textup{d}x \leq \| \varrho \|_{L^q(\Omega)} \left\| \nabla_x(-\Delta_x)^{-1}\big( \mathbbm{1}_{\varrho\geq \bar{\varrho}} (\varrho-\widetilde{\varrho}) \big) \right\|_{L^{\gamma^*}(\Omega; \mathbb{R}^3)}^2,
	\end{equation*}
	with $\gamma^*$ is the Sobolev conjugate of $\gamma$ and
	\begin{equation} \label{esponent varrho}
	q= \frac{3\gamma}{5\gamma -6}.
	\end{equation}
	As $\varrho \in C_{\textup{weak}}([0,T]; L^1 \cap L^{\gamma}(\Omega))$, in order to guarantee that $\varrho(t,\cdot) \in L^q$, one should check that
	\begin{equation} \label{condition on the exponent}
	1\leq q=\frac{3\gamma}{5\gamma -6} \leq \gamma,
	\end{equation}
	which provides the restriction
	\begin{equation} \label{condition gamma 2}
	\frac{9}{5} \leq \gamma \leq 3.
	\end{equation}
	The combination of \eqref{condition gamma 1} and \eqref{condition gamma 2} justifies hypothesis \eqref{condition gamma}. We can now apply Lemma \ref{estimate gradient potential} and, in particular, from \eqref{bound gradient potential} we get 
	\begin{equation*}
	\left\| \nabla_x(-\Delta_x)^{-1}\big( \mathbbm{1}_{\varrho\geq \bar{\varrho}} (\varrho-\widetilde{\varrho}) \big) \right\|_{L^{\gamma^*}(\Omega; \mathbb{R}^3)}^2 \lesssim \| \mathbbm{1}_{\varrho\geq \bar{\varrho}} (\varrho-\widetilde{\varrho}) \|_{L^{\gamma}(\Omega)}^2
	\end{equation*}
	where  
	\begin{equation*}
	\| \mathbbm{1}_{\varrho\geq \bar{\varrho}} (\varrho-\widetilde{\varrho}) \|_{L^{\gamma}(\Omega)}^2= \| \mathbbm{1}_{\varrho\geq \bar{\varrho}} (\varrho-\widetilde{\varrho}) \|_{L^{\gamma}(\Omega)}^{2-\gamma}  \| \mathbbm{1}_{\varrho\geq \bar{\varrho}} (\varrho-\widetilde{\varrho}) \|_{L^{\gamma}(\Omega)}^{\gamma}\lesssim c(\bar{\varrho})\| \mathbbm{1}_{\varrho\geq \bar{\varrho}} (\varrho-\widetilde{\varrho}) \|_{L^{\gamma}(\Omega)}^{\gamma},
	\end{equation*}
	with $2-\gamma \geq 0$ from \eqref{condition gamma}. Hence, we may conclude that
	\begin{equation*}
	\int_{\Omega} \varrho | \nabla_x(-\Delta_x)^{-1}\big( \mathbbm{1}_{\varrho\geq \bar{\varrho}} (\varrho-\widetilde{\varrho}) \big) |^2 \ (t,\cdot) \ \textup{d}x \lesssim \int_{\Omega} E(\varrho, \textbf{u} \ | \ \widetilde{\varrho}, \widetilde{\textbf{u}}) (t, \cdot) \  \textup{d}x.
	\end{equation*}
	
	Similarly, on the other side, from H\"{o}lder's inequality and \eqref{bound gradient potential}, we obtain
	\begin{align*}
	\int_{\Omega} \varrho &\  | \nabla_x(-\Delta_x)^{-1}\big( \mathbbm{1}_{\varrho\leq \bar{\varrho}} (\varrho-\widetilde{\varrho}) \big) |^2  \ \textup{d}x\\
	& \leq \| \varrho \|_{L^{\frac{3}{2}}(\Omega)} \left\| \nabla_x(-\Delta_x)^{-1}\big( \mathbbm{1}_{\varrho\leq \bar{\varrho}} (\varrho-\widetilde{\varrho}) \big) \right\|_{L^6(\Omega; \mathbb{R}^3)}^2 \\
	&\lesssim \| \varrho \|_{L^{\frac{3}{2}}(\Omega)} \left\|  \mathbbm{1}_{\varrho\leq \bar{\varrho}} (\varrho-\widetilde{\varrho}) \right\|_{L^2(\Omega)}^2
	\end{align*}
	From the fact that the pressure potential $P$ is strictly convex on the interval $[0, \bar{\varrho}]$ we have
	\begin{equation*}
	\|  \mathbbm{1}_{\varrho\leq \bar{\varrho}} (\varrho-\widetilde{\varrho}) \|_{L^2(\Omega)}^2 \lesssim \int_{\Omega} E(\varrho, \textbf{u} \ | \ \widetilde{\varrho}, \widetilde{\textbf{u}})(t,\cdot) \  \textup{d}x.
	\end{equation*}
	Getting back to \eqref{energy inequality 2}, we finally obtain
	\begin{equation*}
	\int_{\Omega} E(\varrho, \textbf{u} \ | \ \widetilde{\varrho}, \widetilde{\textbf{u}}) (\tau, \cdot) \  \textup{d}x + \int_{0}^{\tau} \int_{\Omega} |\nabla_x (\textbf{u}-\widetilde{\textbf{u}})|^2 \ \textup{d}x \textup{d}t \lesssim \int_{0}^{\tau} \int_{\Omega} E(\varrho, \textbf{u} \ | \ \widetilde{\varrho}, \widetilde{\textbf{u}}) \  \textup{d}x  \textup{d}t,
	\end{equation*}
	and, as a consequence of Gronwall Lemma,
	\begin{align*}
	\mathcal{E}(\varrho, \textbf{u} \ | \ \widetilde{\varrho}, \widetilde{\textbf{u}})(\tau)=& \int_{\Omega} E(\varrho, \textbf{u} \ | \ \widetilde{\varrho}, \widetilde{\textbf{u}}) (\tau, \cdot) \  \textup{d}x =0 \quad \mbox{for any }\tau \in [0,T], \\
	&\int_{0}^{T} \int_{\Omega} |\nabla_x (\textbf{u}-\widetilde{\textbf{u}})|^2 \ \textup{d}x \textup{d}t=0,
	\end{align*}
	which, in particular, implies \eqref{equality weak and strong}.
\end{proof}

\section{Pressure estimate up to the boundary} \label{Pressure estimates up to the boundary}

The range of $\gamma$ in \eqref{condition gamma} can be enlarged finding some proper pressure estimates ``up to the boundary". The idea is to adapt the procedure performed by Feireisl and Petzeltov\'{a} in \cite{FeiPet} for a bounded domain $\Omega$ in the context of an exterior domain. More precisely, our goal is to prove the following result. 

\begin{theorem} \label{pressure estimate}
	Let $\Omega \subset \mathbb{R}^3$ be a Lipschitz exterior domain. Then, for any dissipative weak solution $(\varrho, \textbf{\textup{u}})$ of the Navier-Stokes-Poisson system \eqref{continuity equation}--\eqref{initial conditions} in the sense of Definition \ref{dissipative solution} there exists a positive constant $K$ such that
	\begin{equation} \label{pressure estimate up to the boundary}
	\int_{0}^{T} \int_{\Omega} \varrho^{\gamma+\omega} \ \textup{d}x \textup{d}t \leq K
	\end{equation}
	for any 
	\begin{equation} \label{condition on omega}
	0 <\omega \leq \frac{2}{3} \gamma-1.
	\end{equation}
\end{theorem}

\begin{proof}
	First of all, fix $R >  \textup{diam}(\Omega^c)$, and define $\varphi \in C^1(\mathbb{R}^3)$ such that 
	\begin{equation*}
		\varphi(x)= \begin{cases}
			0 &\mbox{if } x \in \Omega^c, \\
			1 &\mbox{if } x \in B_R^c,
		\end{cases}
	\end{equation*}
	where $A^c:= \mathbb{R}^3 \setminus A$ denotes the complementary of a set $A$ while $B_R$ denotes the ball of radius $R$ and center the origin. Let us now consider the function
	\begin{equation*}
		\bm{\varphi}(t,x) = \psi(t) \varphi(x) \ \nabla_x (-\Delta_x)^{-1} \big[ b\big( \varrho(t,\cdot) \big)\big](x),
	\end{equation*}
	where $\psi \in C_c^1(0,T)$ and 
	\begin{equation*}
	b\in C^1(\mathbb{R}), \quad b(0)=0, \quad b(z)=z^{\omega} \ \mbox{for } z\geq 1.
	\end{equation*}
	Keeping in mind Remark \ref{less regularity test functions}, we can now use $\bm{\varphi}$ as test function in \eqref{weak formulation of the balance of momentum}; indeed, we obtain
	\begin{equation*}
		\int_{\mathbb{R}}\int_{\mathbb{R}^3} \psi \ \varphi \ p(\varrho) \ b(\varrho) \ \textup{d}x \textup{d}t= \sum_{k=1}^{9} I_k
	\end{equation*}
	with 
	\begin{align*}
		I_1 &= 	-\int_{\mathbb{R}} \int_{\mathbb{R}^3}  \psi' \  \varphi \ \varrho \textbf{u} \cdot \nabla_x (-\Delta_x)^{-1} \big[ b(\varrho) \big] \ \textup{d}x \textup{d}t, \\
		I_2&=+  \int_{\mathbb{R}} \int_{\mathbb{R}^3} \psi \ \varphi \ \varrho \textbf{u} \cdot \big(\nabla_x (-\Delta_x)^{-1} \divv_x\big) \big[ \big(b(\varrho) \textbf{u}\big) \big]  \ \textup{d}x\textup{d}t,  \\
		I_3&= +\int_{\mathbb{R}} \int_{\mathbb{R}^3} \psi \ \varphi \ \varrho \textbf{u} \cdot \nabla_x (-\Delta_x)^{-1} \big[ \big( b(\varrho)-\varrho b'(\varrho)\big) \divv_x \textbf{u} \big]  \ \textup{d}x\textup{d}t, \\
		I_4 &= - \int_{\mathbb{R}}\int_{\mathbb{R}^3} \psi \ [(\varrho \textbf{u} \otimes \textbf{u}) \nabla_x \varphi] \cdot \nabla_x (-\Delta_x)^{-1} \big[ b(\varrho) \big] \ \textup{d}x\textup{d}t, \\
		I_5 &= - \int_{\mathbb{R}} \int_{\mathbb{R}^3} \psi \ \varphi \ (\varrho \textbf{u} \otimes \textbf{u}) : \big(\nabla_x (-\Delta_x)^{-1} \nabla_x\big)	\big[ b(\varrho) \big] \ \textup{d}x\textup{d}t, \\
		I_6 &= - \int_{\mathbb{R}}\int_{\mathbb{R}^3} \psi \ p(\varrho) \ \nabla_x \varphi \cdot \nabla_x (-\Delta_x)^{-1} \big[b(\varrho)\big] \ \textup{d}x \textup{d}t, \\
		I_7 &= + \int_{\mathbb{R}}\int_{\mathbb{R}^3} \psi  \  [\mathbb{S}(\nabla_x \textbf{u})\nabla_x \varphi]  \cdot \nabla_x (-\Delta_x)^{-1} \big[ b(\varrho) \big] \ \textup{d}x\textup{d}t, \\
		I_8 &= + \int_{\mathbb{R}}\int_{\mathbb{R}^3} \psi \ \varphi \  \mathbb{S}(\nabla_x \textbf{u}): \big(\nabla_x (-\Delta_x)^{-1} \nabla_x\big) \big[ b(\varrho) \big] \ \textup{d}x\textup{d}t, \\
		I_9 &= - G \int_{\mathbb{R}}\int_{\mathbb{R}^3} \psi \ \varphi \ \varrho \nabla_x \Phi \cdot \nabla_x (-\Delta_x)^{-1} \big[ b(\varrho)  \big] \ \textup{d}x\textup{d}t.
	\end{align*}
	\begin{itemize}
		\item[(i)] Denoting with $q^*$ the Sobolev conjugate of $q$, given by
		\begin{equation*}
		\frac{1}{q^*}= \frac{1}{q}-\frac{1}{3},
		\end{equation*}
		the combination of Sobolev's inequality with the boundedness of the operator $\nabla_x (-\Delta_x)^{-1} \nabla_x$ form $L^p(\mathbb{R}^3)$ onto $L^p(\mathbb{R}^3; \mathbb{R}^{3\times 3})$ for any $1<p<\infty$, cf.  \eqref{boundedness singular operator}, provides 
		\begin{align*}
		\left\|  \nabla_x (-\Delta_x)^{-1} \big[ b(\varrho) (t,\cdot)\big] \right\|_{L^{q_1^*}(\mathbb{R}^3; \mathbb{R}^3)} &\lesssim \left\| \big(\nabla_x (-\Delta_x)^{-1} \nabla_x\big) \big[ b(\varrho) (t,\cdot)\big] \right\|_{L^{q_1}(\mathbb{R}^3; \mathbb{R}^{3\times 3})} \\[0.1cm]
		&\lesssim \left\|  b(\varrho) (t,\cdot) \right\|_{L^{q_1}(\mathbb{R}^3)}= \left\|  \varrho (t,\cdot) \right\|_{L^{q_1\omega}(\Omega)}^{\omega}.
		\end{align*}
		Therefore, from H\"{o}lder's inequality we get 
		\begin{equation} \label{estimate I1}
		\begin{aligned}
		|I_1| &\leq c(\psi ')\int_{0}^{T}\ \| \sqrt{\varrho}(t,\cdot)\|_{L^{2p}(\Omega)} \| \sqrt{\varrho} \textbf{u}(t,\cdot) \|_{L^2(\Omega; \mathbb{R}^3)} \left\|  \nabla_x (-\Delta_x)^{-1} \big[ b(\varrho ) (t,\cdot)\big] \right\|_{L^{q_1^*}(\mathbb{R}^3; \mathbb{R}^3)} \textup{d}t \\
		&\leq c(\psi ') \int_{0}^{T} \ \| \varrho(t,\cdot)\|_{L^p(\Omega)}^{\frac{1}{2}}  \| \sqrt{\varrho} \textbf{u}(t,\cdot) \|_{L^2(\Omega; \mathbb{R}^3)} \left\|  \varrho (t,\cdot) \right\|_{L^{q_1\omega}(\Omega)}^{\omega} \textup{d}t, 
		\end{aligned}
		\end{equation}
		with 
		\begin{equation*} 
			\frac{1}{q_1} = \frac{5}{6}-\frac{1}{2p}.
		\end{equation*}
		As $\varrho \in C_{\textup{weak}}([0,T]; L^1 \cap L^{\gamma}(\Omega))$ and $1\leq q_1^*<\infty$, we must have
		\begin{equation*}
			\frac{6\gamma}{5\gamma -  3} \leq q_1 <3.
		\end{equation*}
		\item[(ii)] Similarly, form \eqref{boundedness singular operator} and H\"{o}lder's inequality we have
		\begin{align*}
		\left\| \big(\nabla_x (-\Delta_x)^{-1} \divv_x\big) \big[\big(b(\varrho) \textbf{u}\big) (t,\cdot) \big] \right\|_{L^{q_2}(\mathbb{R}^3; \mathbb{R}^3)} &\lesssim \left\|  \big(b(\varrho) \textbf{u}\big) (t,\cdot) \right\|_{L^{q_2}(\mathbb{R}^3; \mathbb{R}^3)} \\
		&\lesssim \| \textbf{u}(t,\cdot)\|_{L^6(\Omega; \mathbb{R}^3)} \|b(\varrho)(t,\cdot)\|_{L^{\frac{6q_2}{6-q_2}}(\mathbb{R}^3)} \\
		&\lesssim \left\|\nabla_x \textbf{u} (t,\cdot) \right\|_{L^2(\Omega; \mathbb{R}^{3\times 3})} \| \varrho(t,\cdot)\|_{L^{\frac{6q_2\omega}{6-q_2}}(\Omega)}^{\omega}.
		\end{align*}
		Therefore, we obtain
		\begin{equation} \label{estimate I2}
		\begin{aligned}
		|I_2| &\lesssim  \int_{0}^{T}  \| \varrho(t,\cdot)\|_{L^{p}(\Omega)} \| \textbf{u}(t,\cdot) \|_{L^6(\Omega; \mathbb{R}^3)} \left\| \big(\nabla_x (-\Delta_x)^{-1} \divv_x\big) \big[ \big(b(\varrho) \textbf{u}\big) (t,\cdot) \big] \right\|_{L^{q_2}(\mathbb{R}^3; \mathbb{R}^3)} \textup{d}t \\
		&\lesssim \int_{0}^{T}  \| \varrho(t,\cdot)\|_{L^{p}(\Omega)} \left\|\nabla_x \textbf{u}(t,\cdot)  \right\|_{L^2(\Omega; \mathbb{R}^{3\times 3})} \| \varrho(t,\cdot)\|_{L^{\frac{6q_2\omega}{6-q_2}}(\Omega)}^{\omega}  \textup{d}t,
		\end{aligned}
		\end{equation}
		with 
		\begin{equation*}
		\frac{1}{q_2}= \frac{5}{6}-\frac{1}{p}
		\end{equation*}
		satisfying
		\begin{equation*}
		\frac{6\gamma}{5\gamma-6} \leq q_2 < \infty.
		\end{equation*}
		\item[(iii)] Proceeding as in (i), we have
		\begin{align*}
		&\left\|  \nabla_x (-\Delta_x)^{-1} \big[ \big( b(\varrho)-\varrho b'(\varrho)\big)(t,\cdot) \divv_x \textbf{u} (t,\cdot) \big]  \right\|_{L^{q_3^*}(\mathbb{R}^3; \mathbb{R}^3)} \\[0.1cm]
		&\quad \quad  \quad \quad \lesssim \left\| \big(\nabla_x (-\Delta_x)^{-1}\nabla_x\big) \big[ \big( b(\varrho)-\varrho b'(\varrho)\big)(t,\cdot) \divv_x \textbf{u} (t,\cdot) \big]  \right\|_{L^{q_3}(\mathbb{R}^3; \mathbb{R}^{3\times 3})}\\[0.1cm]
		&\quad \quad \quad \quad \lesssim \left\| \big[b(\varrho)\divv_x \textbf{u} \big] (t,\cdot) \right\|_{L^{q_3}(\mathbb{R}^3)} \\[0.1cm]
		&\quad \quad \quad \quad \leq \left\|\nabla_x \textbf{u} (t,\cdot) \right\|_{L^2(\Omega; \mathbb{R}^{3\times 3})} \| b(\varrho)(t,\cdot)\|_{L^{\frac{2q_3}{2-q_3}}(\mathbb{R}^3)} \\[0.1cm]
		&\quad \quad  \quad \quad = \left\|\nabla_x \textbf{u} (t,\cdot) \right\|_{L^2(\Omega; \mathbb{R}^{3\times 3})} \| \varrho (t,\cdot)\|_{L^{\frac{2q_3\omega}{2-q_3}}(\Omega)}^{\omega},
		\end{align*}
		and
		\begin{equation} \label{estimate I3}
		\begin{aligned}
		|I_3| &\lesssim  \int_{0}^{T}  \| \varrho(t,\cdot)\|_{L^{p}(\Omega)} \| \textbf{u} (t,\cdot)\|_{L^6(\Omega; \mathbb{R}^3)} \left\|   \nabla_x (-\Delta_x)^{-1} \big[ \big( b(\varrho)-\varrho b'(\varrho)\big) \divv_x \textbf{u} (t,\cdot) \big]  \right\|_{L^{q_3^*}(\mathbb{R}^3; \mathbb{R}^3)} \textup{d}t \\
		&\lesssim \int_{0}^{T}  \| \varrho(t,\cdot)\|_{L^{p}(\Omega)} \left\|\nabla_x \textbf{u}(t,\cdot)  \right\|_{L^2(\Omega; \mathbb{R}^{3\times 3})} \| \varrho(t,\cdot)\|_{L^{\frac{2q_3\omega}{2-q_3}}(\Omega)}^{\omega}  \textup{d}t,
		\end{aligned}
		\end{equation}
		with 
		\begin{equation*}
			\frac{1}{q_3}= \frac{7}{6} - \frac{1}{p} 
		\end{equation*}
		satisfying
		\begin{equation*}
			\frac{6\gamma}{7\gamma -6} \leq q_3 < \infty.
		\end{equation*}
		\item[(iv)] From the fact that
			\begin{align*}
				\varrho \textbf{u} &\in L^2(0,T; L^{\frac{6\gamma}{6+\gamma}}(\Omega; \mathbb{R}^3)), \\
				\textbf{u} &\in L^2(0,T; L^6(\Omega; \mathbb{R}^3)),
			\end{align*}
			we can deduce
			\begin{equation*}
				\varrho \textbf{u} \otimes \textbf{u} \in L^1(0,T; L^q(\Omega; \mathbb{R}^{3\times 3})) \quad \mbox{with } q=\frac{3\gamma}{3+\gamma}.
			\end{equation*}
			Therefore, proceeding as in (i) we have
			\begin{equation} \label{estimate I4}
				\begin{aligned}
					|I_4| &\leq c(\nabla_x \varphi)\int_{0}^{T} \| (\varrho \textbf{u} \otimes \textbf{u})(t,\cdot) \|_{L^{\frac{3\gamma}{3+\gamma}}(\Omega; \mathbb{R}^{3\times 3})} \left\| \nabla_x (-\Delta_x)^{-1} \big[ b(\varrho) (t,\cdot)\big] \right\|_{L^{q_4^*}(\mathbb{R}^3; \mathbb{R}^{3} )} \textup{d}t \\
					& \leq c(\nabla_x \varphi) \int_{0}^{T} \| (\varrho \textbf{u} \otimes \textbf{u})(t,\cdot) \|_{L^{\frac{3\gamma}{3+\gamma}}(\Omega; \mathbb{R}^{3\times 3})} \left\| \varrho (t,\cdot) \right\|_{L^{q_4\omega}(\Omega )}^{\omega}\textup{d}t,
				\end{aligned}
			\end{equation}
			with 
			\begin{equation*}
				\frac{1}{q_4}= 1 -\frac{1}{\gamma}.
			\end{equation*}
			\item[(v)] Similarly, we have 
			\begin{equation} \label{estimate I5}
			\begin{aligned}
			|I_5| &\lesssim \int_{0}^{T} \| (\varrho \textbf{u} \otimes \textbf{u})(t,\cdot) \|_{L^{\frac{3\gamma}{3+\gamma}}(\Omega; \mathbb{R}^{3\times 3})} \left\| \big(\nabla_x (-\Delta_x)^{-1}\nabla_x\big) \big[ b(\varrho) (t,\cdot)\big] \right\|_{L^{q_4^*}(\mathbb{R}^3; \mathbb{R}^{3\times 3} )} \textup{d}t \\
			& \lesssim \int_{0}^{T} \| (\varrho \textbf{u} \otimes \textbf{u})(t,\cdot) \|_{L^{\frac{3\gamma}{3+\gamma}}(\Omega; \mathbb{R}^{3\times 3})} \left\| \varrho (t,\cdot) \right\|_{L^{q_4^*\omega}(\Omega; \mathbb{R}^{3\times 3} )}^{\omega}\textup{d}t.
			\end{aligned}
			\end{equation}
			\item[(vi)] Noticing that 
			\begin{equation*}
				\supp (\nabla_x \varphi) \subset \Omega_R := \Omega \cap B_R,
			\end{equation*}
			and using the Sobolev embedding 
			\begin{equation*}
				W^{1,q_5}(\Omega_R) \hookrightarrow L^{\infty}(\Omega_R) \quad \mbox{with } q_5 >3
			\end{equation*} 
			to deduce 
			\begin{align*}
				\left\| \nabla_x (-\Delta_x)^{-1} \big[ b(\varrho) (t,\cdot)\big] \right\|_{L^{\infty}(\Omega_R; \mathbb{R}^{3})} &\lesssim \left\| \big(\nabla_x (-\Delta_x)^{-1} \nabla_x\big)\big[ b(\varrho) (t,\cdot)\big] \right\|_{L^{q_5}(\Omega_R; \mathbb{R}^{3\times 3})} \\
				& \leq \left\| \big(\nabla_x (-\Delta_x)^{-1} \nabla_x\big)\big[ b(\varrho) (t,\cdot)\big] \right\|_{L^{q_5}(\mathbb{R}^3; \mathbb{R}^{3\times 3})} \\
				& \leq \left\| \varrho (t,\cdot) \right\|_{L^{q_5\omega}(\Omega )}^{\omega},
			\end{align*}
			from H\"{o}lder's inequality we have
			\begin{equation} \label{estimate I6}
				\begin{aligned}
					|I_6| &\leq c(\nabla_x \varphi) \int_{0}^{T} \| p(\varrho)(t,\cdot) \|_{L^1(\Omega)} \left\| \nabla_x (-\Delta_x)^{-1} \big[ b(\varrho) (t,\cdot)\big] \right\|_{L^{\infty}(\Omega_R; \mathbb{R}^{3} )} \textup{d}t \\
					&\leq c(\nabla_x \varphi) \int_{0}^{T} \| p(\varrho)(t,\cdot) \|_{L^1(\Omega)} \left\| \varrho (t,\cdot) \right\|_{L^{q_5\omega}(\Omega )}^{\omega} \textup{d}t.
				\end{aligned}
			\end{equation}
			\item[(vii)]  From H\"{o}lder's inequality we have
			\begin{equation} \label{estimate I7}
				\begin{aligned}
					|I_7| &\leq c(\nabla_x \varphi) \int_{0}^{T} \| \mathbb{S}(\nabla_x \textbf{u}) (t,\cdot) \|_{L^2(\Omega; \mathbb{R}^{3\times 3 })} \left\|  \nabla_x (-\Delta_x)^{-1} \big[ b(\varrho) (t,\cdot)\big] \right\|_{L^2(\mathbb{R}^3; \mathbb{R}^3 )} \textup{d}t \\
					&\leq c(\nabla_x \varphi) \int_{0}^{T} \| \mathbb{S}(\nabla_x \textbf{u}) (t,\cdot) \|_{L^2(\Omega; \mathbb{R}^{3\times 3 })} \left\| \varrho  (t,\cdot) \right\|_{L^{\frac{6}{5}\omega}(\Omega)}^{\omega} \textup{d}t.
				\end{aligned}
			\end{equation}
			\item[(viii)] Similarly,
			\begin{equation} \label{estimate I8}
			\begin{aligned}
				|I_8| &\lesssim \int_{0}^{T} \| \mathbb{S}(\nabla_x \textbf{u}) (t,\cdot) \|_{L^2(\Omega; \mathbb{R}^{3\times 3 })} \left\| \big(\nabla_x (-\Delta_x)^{-1} \nabla_x\big) \big[ b(\varrho) (t,\cdot)\big] \right\|_{L^2(\mathbb{R}^3; \mathbb{R}^{3\times 3} )} \textup{d}t \\
				&\leq \int_{0}^{T} \| \mathbb{S}(\nabla_x \textbf{u}) (t,\cdot) \|_{L^2(\Omega; \mathbb{R}^{3\times 3 })} \left\| \varrho  (t,\cdot) \right\|_{L^{2\omega}(\Omega)}^{\omega} \textup{d}t.
			\end{aligned}
			\end{equation}
			\item[(ix)] From Sobolev's inequality and the boundedness of the operator $\nabla_x (-\Delta_x)^{-1} \nabla_x$ from $L^p(\Omega)$ onto $L^p(\Omega; \mathbb{R}^{3\times 3})$ for any $1<p<\infty$, cf. \eqref{boundedness singular operator}, we have
			\begin{align*}
				\| \nabla_x (-\Delta_x)^{-1} (\varrho+g)(t,\cdot) \|_{L^{\gamma^*}(\Omega; \mathbb{R}^3)} &\lesssim \| \nabla_x (-\Delta_x)^{-1} \nabla_x (\varrho+g)(t,\cdot) \|_{L^{\gamma}(\Omega; \mathbb{R}^{3\times 3})} \\
				& \lesssim \| (\varrho+g)(t,\cdot) \|_{L^{\gamma}(\Omega)},
			\end{align*}
			Therefore, from H\"{o}lder's inequality we get 
			\begin{equation} \label{estimate I9}
				\begin{aligned}
				|I_9| &\lesssim \int_{0}^{T} \| 	\varrho(t,\cdot)\|_{L^{p}(\Omega)} \| \nabla_x (-\Delta_x)^{-1} (\varrho+g)(t,\cdot) \|_{L^{\gamma^*}(\Omega; \mathbb{R}^3)}  \left\|  \nabla_x (-\Delta_x)^{-1} \big[ b(\varrho) (t,\cdot)\big] \right\|_{L^{q_5^*}(\mathbb{R}^3; \mathbb{R}^3)} \textup{d}t \\
				&\lesssim \int_{0}^{T}  \| 	\varrho(t,\cdot)\|_{L^p(\Omega)} \| (\varrho+g)(t,\cdot) \|_{L^{\gamma}(\Omega)} \left\|  \varrho (t,\cdot) \right\|_{L^{q_5\omega}(\Omega)}^{\omega} \textup{d}t,
				\end{aligned}
			\end{equation}
			with 
			\begin{equation*} 
			\frac{1}{q_5} = \frac{5}{3}-\frac{1}{\gamma}-\frac{1}{p},
			\end{equation*}
			satisfying
			\begin{equation*}
			\frac{3\gamma}{5\gamma-6} \leq q_5  \leq \frac{3\gamma}{2\gamma-3}
			\end{equation*}
	\end{itemize}
	
	At this point, in order to guarantee boundedness of the  integrals \eqref{estimate I1}--\eqref{estimate I9} we must require $\alpha \leq \gamma$ in every norm of the type
	\begin{equation*}
	\| \varrho(t,\cdot)\|_{L^{\alpha }(\Omega)},
	\end{equation*}
	and thus
	\begin{itemize}
		\item[-] from \eqref{estimate I1}, we recover: $0<\omega \leq \frac{5}{6} \gamma -\frac{1}{2}$;
		\item[-] from \eqref{estimate I2}, \eqref{estimate I3}, \eqref{estimate I5}, we recover: $0 <\omega \leq \frac{2}{3} \gamma-1$;
		\item[-] from \eqref{estimate I4}, we recover: $0 <\omega \leq \gamma-1$;
		\item[-] from \eqref{estimate I6}, \eqref{estimate I7}, \eqref{estimate I8}, we recover: $0<\omega \leq \min \left\{ \frac{\gamma}{3}, \frac{\gamma}{2}, \frac{5}{6}\gamma \right\} = \frac{\gamma}{3}$;
		\item[-] from \eqref{estimate I9}, we recover: $0< \omega \leq \frac{5}{3} \gamma -2$.
	\end{itemize}
	Since 
	\begin{equation*}
	\frac{2}{3} \gamma-1 \leq \min \left\{ \frac{5}{6} \gamma -\frac{1}{2}, \ \gamma-1, \ \frac{\gamma}{3}, \ \frac{5}{3} \gamma -2\right\}, 
	\end{equation*}
	whenever $1\leq \gamma \leq 3$, it is enough to consider \eqref{condition on omega}. Summing the estimates \eqref{estimate I1}--\eqref{estimate I9}, we get
	\begin{equation*}
	\int_{0}^{T} \int_{\Omega} \psi \ \varphi \ \varrho^{\gamma+\omega} \ \textup{d}x\textup{d}t \lesssim 1,
	\end{equation*}
	and therefore, it is enough to let $\psi, \ \varphi \rightarrow 1$ to obtain \eqref{pressure estimate up to the boundary}.
\end{proof}

\begin{remark}
	Instead of the singular operator $\nabla_x (-\Delta_x)^{-1}$, to prove Theorem \ref{pressure estimate} we could have used the Bogovskii operator $\mathfrak{B}$, which can be interpreted as the inverse of $\divv_x$. More precisely, if we consider the equation
	\begin{equation*}
		\divv_x \textbf{v} =f 
	\end{equation*}
	it has been proved that it admits a solution operator $\mathfrak{B}: f \mapsto \textbf{v}$, bounded from $L^p(\Omega)$ onto $D_0^{1,p}(\Omega; \mathbb{R}^3)$ for any $1<p<\infty$ and any locally Lipschitz exterior domain $\Omega \subset \mathbb{R}^3$, see Galdi \cite{Gal}, Theorem III.3.6. However, in order to get an analogous of estimate \eqref{estimate I2}, we would have needed an additional requirement that if $f= \divv_x \textbf{g}$ for some $\textbf{g}\in L^r(\Omega; \mathbb{R}^3)$, with $\textbf{g}\cdot \textbf{n}|_{\partial \Omega}=0$, then
	\begin{equation*}
	\| \mathfrak{B}[f] \|_{L^r(\Omega; \mathbb{R}^3)} \leq c(p,r,\Omega) \| \textbf{g} \|_{L^r(\Omega; \mathbb{R}^3)};
	\end{equation*}
	this result is known to be true for bounded domains (see, for instance, \cite{Gal}, Theorem III.3.4) but not for unbounded domains, to the best of the author's knowledge.
\end{remark}

Theorem \ref{pressure estimate} in particular implies that 
\begin{equation} \label{better estimate density}
\varrho \in L^{\gamma + \omega} ((0,T) \times \Omega)  \quad \mbox{with} \quad \omega= \frac{2}{3}\gamma-1,
\end{equation}
and therefore the range of $\gamma$ in \eqref{condition gamma} can be slightly improved. More precisely, we have the following last result.

\begin{corollary} \label{corollary}
	There exists $\gamma^* < \frac{7}{4}$ such that Theorem \ref{weak-strong uniqueness} holds with condition \eqref{condition gamma} replaced by 
	\begin{equation*}
	\gamma^* \leq \gamma \leq 2.
	\end{equation*} 
\end{corollary}
\begin{proof}
	Due to condition \eqref{better estimate density}, we only have to replace \eqref{condition on the exponent} in the proof of Theorem \eqref{weak-strong uniqueness} with
	\begin{equation*}
	1 \leq \frac{3\gamma}{5\gamma-6} \leq \gamma+ \omega= \frac{5}{3} \gamma -1,
	\end{equation*}
	which provides
	\begin{equation} \label{new condition gamma}
	\gamma^* \leq \gamma \leq 3, \quad \mbox{with } \gamma^*= \frac{27}{25}+ \frac{3 \sqrt{31}}{25}< \frac{7}{4}.
	\end{equation}
	Unifying conditions \eqref{new condition gamma} and \eqref{condition gamma 1}, we get the claim.
\end{proof}

\bigskip

\centerline{\bf Acknowledgement}

The author wishes to thank Prof. Eduard Feireisl for the helpful advice and discussions.

\end{document}